\documentclass[centertags,12pt]{amsart}

\usepackage{amssymb}
\usepackage{latexsym}
\usepackage{amsfonts}
\usepackage{amssymb}
\usepackage{graphicx} 
\usepackage{pdflscape} 
\usepackage{diagbox} 
\usepackage{hyperref}

\makeatletter
\renewcommand{\subsection}{\@startsection
{subsection}{2}{0mm}{\baselineskip}{-0.25cm}
{\normalfont\normalsize\em}}
\makeatother

\newtheorem{theorem}{Theorem}
\newtheorem{proposition}[theorem]{Proposition}

{\theoremstyle{definition}

\newtheorem{example}[theorem]{Example}

}

{\theoremstyle{remark}
\newtheorem{remark}[theorem]{Remark}

}

\def\Ap{{\rm Ap}}
\def\Kunz{{\rm Kunz}}

\def\N {\mathbb{N}}
\def\Z {\mathbb{Z}}

\def\l {\ell}

\title[A generalization of a theorem about gapsets with depth at most three]{A generalization of a theorem about gapsets with depth at most three}

\author[M. Bernardini]{Matheus Bernardini}
\address{Universidade de Bras\'{i}lia, Faculdade do Gama, Bras\'{i}lia, DF, Brazil}
\email{matheusbernardini@unb.br}

\author[P. Melo]{Patrick Melo}
\address{Universidade de Bras\'{i}lia, Faculdade do Gama, Bras\'{i}lia, DF, Brazil}
\email{patrickmelo13@gmail.com}

\thanks{{\em 2020 Math. Subj. Class.}: Primary 20M14, 05A15;  Secondary 05A19}


\begin{document}


\begin{abstract} 
In this paper, we provide a generalization of a theorem proved by Eliahou and Fromentin, which exhibit a remarkable property of the sequence $(n'_g)$, where $n'_g$ denotes the number of gapsets with genus $g$ and depth at most $3$.
\end{abstract}

\maketitle

\section{Introduction}\label{s1}

Let $\N$ be the set of positive integers and $\N_0 = \N \cup \{0\}$. A \textit{numerical semigroup} $S$ is a submonoid of $\N_0$, equipped with the usual addition, such that $G(S):=\N_0\setminus S$, the set of {\em gaps} of $S$, is finite. In \cite{EF}, the concept of a gapset was formally introduced in the following way. A \textit{gapset} is a finite set $G \subset \N$ satisfying the following property: let $z \in G$ and write $z = x + y$, with $x$ and $y \in \N$; then $x \in G$ or $y \in G$. There is a bijective map between the set of numerical semigroups and the set of gapsets given by $S \mapsto \N_0 \setminus S$. Thus, a gapset is the set of gaps of some numerical semigroup and one can define some invariants of a gapset by using the invariants of its complement in $\N_0$. For instance, the \textit{genus}, the \textit{multiplicity}, the \textit{conductor} and the \textit{depth} of a gapset $G$ are $g(G) = \# G, m(G) := \min\{s \in \N: s \notin G\}$, $c(G) := \min\{s \in \N: s + n \notin G, \forall n \in \N_0\}$ and $q(G) := \left \lceil \frac{c(G)}{m(G)} \right \rceil$, respectively. Observe that, if $G$ is nonempty, then $c(G) = \max(G) + 1$.

A central problem in numerical semigroup theory is Bras-Amorós' conjecture, which was originally stated with three items (see \cite{Amoros1}). It consists of understanding the behaviour of the sequence $(n_g)$\footnote{This sequence is registered as A007323 at OEIS.}, where $n_g$ denotes the number of gapsets (or numerical semigroups) with a fixed genus $g$. Two of the items of the conjecture are about asymptotic behaviour of $n_g$ ((1) $\lim_{g \to \infty} \frac{n_{g}}{n_{g-1}} = \varphi$, the golden ratio, and (2) $\lim_{g \to \infty} \frac{n_{g-1} + n_{g-2}}{n_{g}} = 1$; observe that $(1) \Rightarrow (2)$) and were proved by Zhai \cite{Zhai}. The key ingredient of the proof was observing that ``almost all" gapsets with a fixed genus have depth at most $3$. The only item of Bras-Amorós' conjecture that remains as an open problem is ``is it true that $n_g + n_{g+1} \leq n_{g+2}$, for all $g$?". Also, a weaker version of this conjecture, namely ``is $(n_g)$ a non-decreasing sequence?", is still an open problem. See \cite{Kaplan2} for more details.

Eliahou and Fromentin \cite{EF} studied a related problem to this one. They proved a remarkable property about the behaviour of the sequence $(n'_g)$, where $n'_g$ denotes the number of gapsets with a fixed genus $g$ and depth at most 3 (we denote by $\Gamma'(g)$ the set of gapsets with genus $g$ and depth at most $3$). It is stated as follows.

\begin{theorem}[\cite{EF}]
Let $(n'_g)$ be the sequence of the number of gapsets with genus $g$ and depth at most $3$. Then $n'_{g-1} + n'_{g-2} \leq n'_g \leq n'_{g-1} + n'_{g-2} + n'_{g-3}$.
\label{EliFro}
\end{theorem}

In that paper, the authors prove Theorem 1 in two parts. The first inequality is a consequence of the construction of two injective maps: one of them has domain given by the set of gapsets with genus $g-1$ and depth at most $3$ and the other one has has domain given by the set of gapsets with genus $g-2$ and depth at most $3$; both of them have codomain given by the set of gapsets with genus $g$ and depth at most $3$ and the images of those maps are disjoint sets. For the second inequality, the authors split the set of gapsets with genus $g$ and depth at most $3$ into three disjoint parts, $\Gamma'_1(g) \ \dot\cup \ \Gamma'_2(g) \ \dot\cup \ \Gamma'_3(g)$, and they construct three injective maps, where $\Gamma'_k(g)$ is mapped into the set of gapsets with genus $g-k$ and depth at most $3$, for $k \in \{1, 2, 3\}$.

In this paper, we present a generalization of Theorem \ref{EliFro}. Here, we construct injective maps inspired by those that have been considered by Eliahou and Fromentin; the main difference is that we analyse how the Kunz coordinates are modified under those maps. In section 2, we introduce the Kunz coordinates for gapsets and we prove that the depth of a gapset is its largest coordinate; moreover, we introduce the notion of level of a gapset. In section 3, we prove a generalization of Theorem \ref{EliFro} by using the Kunz coordinates of a gapset. We prefer the gapset language to write this paper, but it could also be done using numerical semigroup theory.

Here, we observe that similar ideas of a particular case of Theorem \ref{EliFroGen} have been considered, independently, by Zhu \cite{Zhu} (see \textit{Proof of Theorem} 1.4 of his paper).

\section{The Kunz coordinates and the level of a gapset}

In this section, we introduce the Ap\'{e}ry set and the Kunz coordinates of a gapset $G$. Those definitions arise in a natural way when we consider the numerical semigroup $\N_0 \setminus G$.

Let $S$ be a numerical semigroup with multiplicity $m$ (which is the least non zero element of $S$). The Ap\'{e}ry set of $S$ (on $m$) is defined as $\Ap(S) = \{w_0, w_1, \ldots, w_{m-1}\}$, where $w_i = \min\{s \in S: s \equiv i \pmod{m}\}$, for $i \in [0, m-1]$. Notice that $w_0 = 0$ and there is a $k_i \in \N$ such that $w_i = mk_i + i$. The Kunz coordinates of $S$ (on $m$) are \linebreak $\Kunz(S) = (k_1, k_2, \ldots, k_{m-1})$, where $k_i = (w_i - i)/m$. Now, we transfer this terminology to a gapset $G$ with multiplicity $m$; the Ap\'{e}ry set of $G$ (on $m$) and the Kunz coordinates of $G$ (on $m$) are given by $\Ap(G) := \Ap(\N_0 \setminus G)$ and $\Kunz(G) := \Kunz(\N_0 \setminus G)$, respectively. 

\begin{proposition}
Let $G$ be a gapset with multiplicity $m$, such that \linebreak $\Ap(G) = \{w_0, w_1, \ldots, w_{m-1}\}$. Then $w_i = m + \max\{z \in G: z \equiv i \pmod{m}\}$ if $i \neq 0$ and $w_0 = 0$. Moreover, if $\Kunz(G) = (k_1, k_2, \ldots, k_{m-1})$, then \linebreak $k_i = \#\{z \in G: z \equiv i \pmod{m}\}$ for all $i \in [1, m-1] \cap \Z$.
\label{Apery}
\end{proposition}

\begin{proof}
First, observe that $w_0 = 0$. Also, by the definition of $w_i$, we conclude that \linebreak $G \cap \{n \in \N: n \equiv i \pmod{m}\} = \{i, i + m, \ldots, i + (k_i - 1)m\}$, since $mx + i \notin G$ for all $x \geq k_i$. Hence, $w_i = i + (k_i -1)m + m$ and the result follows. Moreover, the set $G \cap \{n \in \N: n \equiv i \pmod{m}\}$ has $k_i$ elements.
\end{proof}

\begin{example}
The gapset $G = \{1, 2, 4, 5, 8, 11\}$ has multiplicity 3. In this case, \linebreak $\Ap(G) = \{0, 7, 14\}$ and $\Kunz(G) = (2,4)$.
\end{example}

One can characterize numerical semigroups with multiplicity $m$ in terms of their Kunz coordinates, namely $\Kunz(S) = (k_1, k_2, \ldots, k_{m-1})$. As a matter of fact, a tuple in $\N^{m-1}$ is the Kunz coordinates of some numerical semigroup with multiplicity $m$ if, and only if, it satisfies the following system of inequalities (cf. \cite{Rosales}):

\begin{equation}
\begin{cases}
X_i \in \N \\
X_i + X_j \geq  X_{i+j}, \hspace{1.2cm} \text{ for } 1 \leq i \leq j \leq m-1; \ i + j < m \\
X_i + X_j + 1 \geq X_{i+j-m}, \text{ for } 1 \leq i \leq j \leq m-1; \ i + j > m.
\end{cases}
\label{system}
\end{equation}

In particular, we conclude that the Kunz coordinates of a gapset with multiplicity $m$ also must satisfy the system (\ref{system}).

\begin{example}
The tuple $(2,3,3,1)$ lists the Kunz coordinates of the gapset $\{1, 2, 3, 4, 6, 7, 8, 12, 13\}$. However, the tuple $(1,3,3,2)$ does not list the Kunz coordinates of any gapset, since $k_1 + k_1 = 2 < 3 = k_2$. 
\label{naoexemplo}
\end{example}

The \textit{canonical partition} of a gapset $G$ was introduced in \cite{EF} as 
$$G = G_0 \cup G_1 \cup \ldots \cup G_{q-1},$$ 
where $G_0 = [1, m-1] \cap \Z$ and $G_{i+1} \subseteq G_i + m$. Basically, it is a clipping of the set $G$ into $q$ parts where each part lies in an interval of integers of the type $[am + 1, (a+1)m - 1] \cap \Z$, for some $a \in \N$.

\begin{proposition}
Let $G$ be a gapset, with $\Kunz(G) = (k_1, k_2, \ldots, k_{m-1})$, genus $g$ and depth $q$. Then $g = \sum_{i=1}^{m-1} k_i$ and $q = \max\{k_i: 1 \leq i \leq m-1\}$.
\label{genus_depth_kunz}
\end{proposition}

\begin{proof}
By Proposition \ref{Apery}, $G \cap \{n \in \N: n \equiv i \pmod{m}\}$ has $k_i$ elements, if $i \neq 0$ and $G \cap m\N = \emptyset$. Thus the formula for the the genus can be obtained by summing up the coordinates of $\Kunz(G)$. Consider the canonical partition of $G$, $G_0 \cup G_1 \cup \cdots \cup G_{q-1}$ and let $x \in G_{q-1}$ (that exists, since the depth of $G$ is $q$). In particular, $x - \l m \in G$ for all $\l \in \N_0$ such that $x -\l m > 0$ and there is exactly one element in each $G_i$ that is congruent to $x$ modulo $m$. Using Proposition \ref{Apery} again, we conclude that the depth of $G$ (which is also the quantity of parts of the canonical partition of $G$) coincides with the quantity of elements that are congruent to $x$ modulo $m$. Hence, $q = k_t$, where $t = \min\{x - \l m: x - \l m > 0, \l \in \Z\}$, i.e., $t$ is the remainder when $x$ is divided by $m$.
\end{proof}

Now, we introduce a new parameter of a gapset, which is its level. Let $S$ be a numerical semigroup with multiplicity $m > 1$. The \textit{ratio} of $S$, $r(S)$, can be defined as \linebreak $\min\{s \in S: s \not\equiv 0 \pmod {m}\}$ (see \cite{GS-R} for more details). Thus, the ratio of a gapset $G$ with multiplicity $m$ is $r(G) := \min\{x \notin G: x \not\equiv 0 \pmod {m}\}$. Finally, the \textit{level} of a gapset $G$ is $\lambda(G) := \left \lfloor \frac{r(G)}{m(G)} \right \rfloor$. The next result allows us to relate the level of a gapset with its Kunz coordinates.

\begin{proposition}
Let $G$ be a gapset, with $\Kunz(G) = (k_1, k_2, \ldots, k_{m-1})$ and level $\lambda$. Then $\lambda = \min\{k_i: 1 \leq i \leq m-1\}$.
\label{level}
\end{proposition}

\begin{proof}
Notice that the multiplicity of $G$ is $m$. Let $r$ and $\lambda$ be the ratio and the level of $G$, respectively. By the definition of $r$, we have that $\{r - km: r - km > 0, k \in \N\} \subset G$. Moreover, $r$ is the smallest positive integer outside $G$ that does not divide $m$. Hence, the cardinality of $\{r - km: r - km > 0, k \in \N\}$ coincides with the least Kunz coordinate of $G$. Since $\{r - km: r - km > 0, k \in \N\} = [1, \frac{r}{m}] \cap \Z = [1, \lambda] \cap \Z$, the result follows.
\end{proof}

For $\l \in \N$, consider $\Gamma'(g, \l)$ as the set of gapsets with genus $g$ whose Kunz coordinates lie in $[\l, 2\l + 1] \cap \Z$. This set can be characterized in terms of the level and the depth of a gapset, as follows:
$$\Gamma'(g, \l) = \{G \in \Gamma(g): \l \leq \lambda(G) \leq q(G) \leq 2\l + 1\},$$
and we denote by $n'_{g,\l}$ its cardinality. In particular, $\Gamma'(g,1) = \Gamma'(g)$ and $n'_{g,1} = n'_{g}$. Moreover, $\Gamma'(g, \l) = \emptyset$, if $\l > g$ and $n'_{g, \l} = 0$ in this case. We deal the numbers $n'_{g,\l}$ in the next section.

\section{A generalization of Theorem \ref{EliFro}}

In this section, we present the main result of this paper, which is a generalization of Theorem \ref{EliFro}.

\begin{theorem}
Let $g$ and $\l$ be positive integers and $(n'_{g,\l})$ be the sequence of the number of gapsets with genus $g$ whose Kunz coordinates lie in $[\l, 2\l + 1] \cap \Z$. Then 
$$\sum_{i = \l}^{2\l} n'_{g - i, \l} \leq n'_{g,\l} \leq \sum_{i = \l}^{2\l+1}  n'_{g - i, \l}.$$
\label{EliFroGen}
\end{theorem}

\begin{proof}
Let $\Gamma'(g, \l)$ be the set of gapsets with genus $g$ whose Kunz coordinates lie in $[\l, 2\l + 1] \cap \Z$. First, we identify the gapsets of $\Gamma'(g, \l)$ with tuples whose coordinates lie in $[\l, 2\l + 1] \cap \Z$, the sum of its coordinates is $g$ (cf. Proposition \ref{genus_depth_kunz}) and that satisfies the system (\ref{system}). 

Now we prove the first inequality by considering the following functions: for \linebreak $z \in [\l, 2\l] \cap \Z$, consider $f_z: \Gamma'(g-z, \l) \to \Gamma'(g, \l)$, which are described by the Kunz coordinates. If $G$ has multiplicity $m$ and $\Kunz(G) = (k_1, k_2, \ldots, k_{m-1})$, then we define \linebreak $\Kunz(f_z(G)) = (k_1, k_2, \ldots, k_{m-1}, z)$. Notice that all the functions are injective and their images are disjoint sets since the last coordinate of $\Kunz(f_z(G))$ and $\Kunz(f_w(G))$ are different, if $z \neq w$. Also, if $G \in \Gamma'(g-z, \l)$ and $\Kunz(G) = (k_1, \ldots, k_{m-1})$, then the genus of $G$ is $\sum_{i=1}^{m-1} k_i = g - z$ (cf. Proposition \ref{genus_depth_kunz}), $k_i \in [\l, 2\l+1] \cap \Z$, for all $i$, the genus of $f_z(G)$ is $g$ and its coordinates lie in $[\l, 2\l + 1] \cap \Z$. It remains to show that $(k_1, \ldots, k_{m-1}, k_m)$, where $k_m = z$, satisfies the system (\ref{system}). If $1 \leq i \leq j \leq m-1$ and $i + j < m$, then $k_i + k_j \geq k_{i+j}$ (by hypothesis). If $1 \leq i \leq j \leq m-1$ and $i + j = m$, then $k_i + k_j \geq 2\l$ and $z \leq 2\l$; thus $k_i + k_j \geq k_{i+j} = z$. Finally, if $1 \leq i \leq j \leq m$ and $i + j > m+1$, then $k_i + k_j + 1 \geq 2\l + 1$ and $k_{i+j - m} \leq 2\l + 1$; thus $k_i + k_j + 1 \geq k_{i+j - m}$.

Now we prove the second inequality. Let $z \in [\l, 2\l + 1] \cap \Z$ and denote by $\Gamma'_z(g,\l)$ the set of gapsets with genus $g$ whose Kunz coordinates lie in $[\l, 2\l + 1] \cap \Z$ and its last Kunz coordinate is equal to $z$. Notice that the set $\Gamma'(g, \l)$ is the (disjoint) union of the sets $\Gamma'_z(g,\l)$.  For $z \in [\l, 2\l + 1] \cap \Z$, consider the function $h_z: \Gamma_z'(g, \l) \to \Gamma'(g-z, \l)$, which is described by the Kunz coordinates. If $\Kunz(G) = (k_1, k_2, \ldots, k_{m-1}, k_m)$, with $k_m = z$, then we define $\Kunz(h_z(G)) = (k_1, k_2, \ldots, k_{m-1})$. Notice that this function is injective and if $G \in \Gamma_z'(g, \l)$, then the genus of $h_z(G)$ is $g - z$  (cf. Proposition \ref{genus_depth_kunz}). It remains to show that $(k_1, \ldots, k_{m-1})$ satisfy the system (\ref{system}). If $1 \leq i \leq j \leq m-1$ and $i + j < m$, then $k_i + k_j \geq k_{i+j}$ (by hypothesis). Finally, if $1 \leq i \leq j \leq m-1$ and $i + j > m$, then $k_i + k_j + 1 \geq 2\l + 1$ and $k_{i+j-m} \leq 2\l + 1$; thus $k_i + k_j + 1 \geq k_{i+j - m}$ and we are done.
\end{proof}

Some values of $n'_{g, \l}$ are listed at Table \ref{g,l}. We observe that if $g$ is even, then $n'_{g, \frac{g}{2}} = 2$ and if $\frac{g}{2} < \l \leq g$, then $n'_{g, \l} = 1$, since $\Gamma'(g, \frac{g}{2}) = \{G: \Kunz(G) = (\frac{g}{2}, \frac{g}{2}) \text{ or } \Kunz(G) =  (g)\}$ and if $\frac{g}{2} < \l \leq g$, then $\Gamma'(g, \l) = \{G: \Kunz(G) =  (g)\}$. Notice that there are gapsets that do not lie in any of the sets $\Gamma'(g, \l)$. It occurs with the gapset $G$ such that $\Kunz(G) = (4, 2, 1)$, for example. Moreover, there are gapsets that lie in several of the sets $\Gamma'(g, \l)$. For instance, if $\Kunz(G) = (g)$, then $G \in \Gamma'(g, \l)$, for all $\l \in \left[\left \lfloor \frac{g}{2} \right \rfloor, g \right]$.

\setlength{\tabcolsep}{0.75em} 
{\renewcommand{\arraystretch}{1.4}
\begin{table}[h!]
\caption{A few values for $n'_{g,\l}$.}
\label{g,l}
\begin{tabular}{|c|c c c c c c c c c c|}
 \hline
\diagbox[height=0.9cm]{$g$}{$\l$} & $1$ & $2$ & $3$ &$4$ & $5$ & $6$ & $7$ & 8 & 9 & 10 \\
\hline
1 & 1 &  &  &  &  &  &  &  & & \\
\hline
2 & 2 & 1 &  &  &  &  &  &  & & \\
\hline
3 & 4 & 1 & 1 &  &  &  &  &  & & \\
\hline
4 & 6 & 2 & 1 & 1 &  &  &  &  & & \\
\hline
5 & 11 & 3 & 1 & 1 & 1 &  &  &  & & \\
\hline
6 & 20 & 4 & 2 & 1 & 1 & 1 &  &  & & \\
\hline
7 & 33 & 6 & 3 & 1 & 1 & 1 & 1 &  & & \\
\hline
8 & 57 & 10 & 3 & 2 & 1 & 1 & 1 & 1 & & \\
\hline
9 & 99 & 14 & 5 & 3 & 1 & 1 & 1 & 1 & 1 & \\
\hline
10 & 168 & 22 & 7 & 3 & 2 & 1 & 1 & 1 & 1 & 1 \\
\hline
\end{tabular}
\end{table}

\begin{remark}
If $\l = 1$ in Theorem \ref{EliFroGen}, then we obtain the statement of Theorem \ref{EliFro}. We observe that all the tuples with coordinates 1, 2 or 3 satisfy the condition \linebreak $X_i + X_j + 1 \geq X_{i+j}$ of the system (\ref{system}). For tuples with at least one coordinate greater than 3, this statement may not be true and it can be a reason for the difficulty of constructing (injective) maps that preserve the gapset property. In particular, the functions considered in this paper cannot be extended to gapsets with fixed genus and depth greater than 3. For instance, if $\Kunz(G) = (4, 2, 1)$, then the tuples obtained by adding $1$ and $2$ as the last coordinate are $(4, 2, 1, 1)$ and $(4, 2, 1, 2)$, respectively, but they are not the Kunz coordinates of any gapset (in both cases, $2k_3 + 1 = 3 < 4 = k_1$). Also, if $\Kunz(G) = (4, 1, 4)$, then the tuple obtained removing the last coordinate is $(4, 1)$, but it does not list the Kunz coordinates of any gapset ($2k_2 + 1 = 3 < 4 = k_1$).
\end{remark}

{\bf Acknowledgment.} We thank the anonymous referee for her/his careful corrections, suggestions and comments that helped to improve this version of the paper, specially for pointing out that the proof of Theorem \ref{EliFro} using Kunz coordinates could be generalized to obtain Theorem \ref{EliFroGen}. We thank Shalom Eliahou for nice discussions on this subject. We also thank the editor (Nathan Kaplan) for his careful suggestions and corrections.


\begin{thebibliography}{99}
    

\bibitem{Amoros1} Bras-Amor\'{o}s, M.: Fibonacci-like behavior of the number of numerical semigroups of a given genus, Semigroup Forum 76, 379 -- 384 (2008)

\bibitem{EF} Eliahou, S., Fromentin, J.: Gapsets and numerical semigroups, Journal of Combinatorial Theory, Series A 169, 105 -- 129 (2020)

\bibitem{GS-R} Garc\'{i}a-S\'{a}nchez, P.A., Rosales, J.C.: Numerical semigroups, Developments in Mathematics vol. 20,  Springer, New York (2009)

\bibitem{Kaplan2} Kaplan, N.: Counting numerical semigroups, The American Mathematical Monthly 124, 862 -- 875 (2017)

\bibitem{Rosales} Rosales, J.C., García-Sanchéz, P.A., García-García, J.I., Branco, M.B.: Systems of inequalities and numerical semigroups, Journal of the London Mathematical Society (2) 65, 611 -- 623 (2002)

\bibitem{Zhai} Zhai, A.: Fibonacci-like growth of numerical semigroups of a given genus, Semigroup Forum 86, \linebreak 634 -- 662 (2013)

\bibitem{Zhu} Zhu, D.: Sub-Fibonacci behavior in numerical semigroup enumeration, arXiv:2202.05755 (2022)

\end{thebibliography}
\end{document}